\documentclass[10pt,a4paper]{article}

\usepackage[english]{babel}
\usepackage{amssymb}
\usepackage{amstext}
\usepackage{dsfont}
\usepackage{amsthm}
\usepackage{paralist}
\usepackage{fancyhdr}
\usepackage{hyperref}
\usepackage{here}
\usepackage{comment}
\usepackage{pgfplots}
\usepackage{titlesec}
\usepackage{mathtools}

\titleformat{\section}
  {\normalfont\LARGE\bfseries}
 {\thesection}{1em}{}
\titleformat{\subsection}
  {\normalfont\Large\bfseries}
  {\thesubsection}{1em}{}
\titleformat{\subsubsection}
  {\normalfont\large\bfseries}
  {\thesubsubsection}{1em}{}

\newtheorem{thm}{Theorem}
\newtheorem{lemma}{Lemma}[section]

\newtheorem{clm}[lemma]{Claim}
\newtheorem{cor}[lemma]{Corollary}

\newtheorem{defn}[lemma]{Definition}

\newcommand{\XSays}[3]{{\color{#2}
      {$\rule[-0.12cm]{0.2in}{0.5cm}$\fbox{
            #1:} }%
      \itshape #3
      \marginpar{\color{#2} #1}%
      \def\comment{#3}\def\empty{}\ifx\comment\empty\else
      {$\rule[0.1cm]{0.3in}{0.1cm}$\fbox{
            end}$\rule[0.1cm]{0.3in}{0.1cm}$} \fi
   }%
}

\numberwithin{equation}{subsection}

\begin{document}
\title{Bootstrap Percolation on  Degenerate Graphs}
\author{Marinus Gottschau\\Technische Universit\"at M\"unchen}
\maketitle
\begin{abstract}
In this paper we focus on $r$-neighbor bootstrap percolation, which is a process on a graph where initially a set $A_0$ of vertices gets infected. Now subsequently, an uninfected vertex becomes infected if it is adjacent to at least $r$ infected vertices. Call $A_f$ the set of vertices that is infected after the process stops. More formally set $A_t:=A_{t-1}\cup \{v\in V: |N(v)\cap A_{t-1}|\geq r\}$, where $N(v)$ is the neighborhood of $v$. Then $A_f=\bigcup_{t>0} A_t$. We deal with finite graphs only and denote by $n$ the number of vertices.\\
We are mainly interested in the size of the final set $A_f$. We present a theorem for degenerate graphs that bounds the size of the final infected set. More precisely for a $d$-degenerate graph, if $r>d$, we bound the size set $A_f$ from above by $(1+\tfrac{d}{r-d})|A_0|$.
\end{abstract}
\section{Introduction}
An $r$-neighbor bootstrap percolation process on some given graph $G=(V,E)$ with vertex set $V$ and edge set $E$ is an discrete time infection process. Initially there is some set of infected vertices. Now at every time step, every vertex that has at least $r$ infected neighbors becomes infected in the next time step. The process was first introduced by Chalupa, Leath and Reich in 1979 in~\cite{origin} and is a simple example for a cellular automaton. It is also closely related to the Glauber dynamics which represent the Ising model at zero-temperature (see \cite{Ising}). Another application one can think of is rumor spreading in a social network. Instead of speaking of infection the literature also uses the term activation. We shall stick to the term infection and  write bootstrap percolation process instead of $r$-neighbor bootstrap percolation process.\\
We call the set of initially infected vertices $A_0$ and the vertices that are infected at the end of the process $A_f$. More formally, we set the vertices which are infected at time $t$ to be $A_t:=A_{t-1}\cup \{v\in V: |N(v)\cap A_{t-1}|\geq r\}$ where $N(v)$ denotes the neighborhood of $v$ and thus $A_f=\bigcup_{t>0} A_t$. Note that there are two commonly used ways of obtaining the set $A_0$. On the one hand, one can infect each vertex independently with a given probability $p$. On the other hand, when dealing with a finite graph, one might want $A_0$ to be of a given size. Here one chooses uniformly at random from all sets that are of the given size.\\
For several graph classes there is already much known about this process. There are a few things that are of interest. First of all one can study the probability of percolation, i.e. the probability that $A_f=V$, depending on $p$, which is the probability with which each vertex independently is initially infected. Let $\alpha\in [0,1]$, define ${p_\alpha(G,r):=\inf\{p: \mathbb{P}[A_f=V]\geq \alpha\}}$. Several authors surveyed the critical probability for $\alpha=\tfrac{1}{2}$, which means that percolation is more likely to occur than no percolation. For example, if the underlying graph is the the $d$-dimensional cube graph $[n]^d$, Aizenman and Lebowitz \cite{grid} showed that for fixed $d$, we have
\begin{align*}
p_{\frac{1}{2}}([n]^d,2)=\Theta\left(\frac{1}{\log n}\right)^{d-1}.
\end{align*}
Later Cerf and Manzo \cite{grid2} generalized this to
\begin{align*}
p_{\frac{1}{2}}([n]^d,r)=\Theta\left(\frac{1}{\log_{(r-1)} n}\right)^{d-r+1},
\end{align*}
where $\log_{(r-1)}$ is an $r$ times iterated logarithm. The exact threshold function is known for $d=r=2$ and was shown by Holroyd in \cite{rd2} to be $\tfrac{\pi^2}{18 \log n}+ o(\tfrac{1}{\log n})$. Balogh et al.~studied in \cite{threshold} the case $d=r=3$ and gave a conjecture on the threshold function for more general parameter choices.\\
Additionally, bootstrap percolation was studied on trees like periodic trees in \cite{periodictrees} and Galton-Watson trees in \cite{trees}, \cite{GWtrees} and \cite{GWtrees1}.\\
Some papers also study the size of minimal percolating sets which are sets of vertices that infect the complete graph, but any proper subset does not. Riedl showed in \cite{riedl} that for a tree on $n$ vertices with $l$ vertices of degree less than $r$, a minimal percolating set $A_0$ is of size
\begin{align*}
\frac{(r-1)n+1}{r}\leq |A_0| \leq \frac{rn+l}{r+1}.
\end{align*}
He also gave an algorithm that computes the size of a smallest percolating set as well as the size of a largest minimal percolating set and also did some work on hypercubes under 2-bootstrap percolation (see \cite{riedl2}).\\
Also other graphs like the well known Erd\H{o}s-R\'enyi random graph have been studied, for example in \cite{Gnp} by Janson et al. Here the authors give a function for the edge probabilities when percolation occurs with high probability, depending on the size of $A_0$. Also they give a criterion for almost sure percolation when the edge probabilities are given, again depending on the size of the initially infected vertex set $A_0$.\\
Another quite interesting parameter is the running time of such a process, which is the time until no new vertex becomes infected, i.e.~the least $t$ such that $A_t=A_{t+1}=A_f$. This parameter has been studied for several graphs like the grid $[n]^2$, where Benevides and Przykucki \cite{time} showed that for $r=2$ the running time is bounded by $\tfrac{13}{18}n^{2}+\mathcal O(n)$. In \cite{time1} Przykucki considered bootstrap percolation on the $d$-dimensional hypercube and proved the time to be at most $\lfloor\tfrac{d^2}{3}\rfloor$ again for $r=2$. Bollob{\'a}s et al.~\cite{time2} analyzed the time of bootstrap percolation on the discrete torus while Janson et al.~\cite{Gnp} gave a time bound for the percolation process on the Erd{\"o}s R{\'e}nyi random graph.\\
In this paper we focus on the size of the infected set $A_f$ at the end of the process for degenerate graphs. We give a result for $r$-neighbor bootstrap percolation when the underlying graph is a degenerate graph and the set $A_0$ is any subset of the vertices. As many graph classes have bounded degeneracy our result also covers many other classes.
\section{Bootstrap percolation on degenerate graphs}
Our result gives a bound on the size of the set $A_f$ of vertices that are infected at the end of the process on a degenerate graph, so let us first define degeneracy.
\begin{defn}
A finite graph $G=(V,E)$ is called $d$-degenerate if every subgraph contains a vertex of degree at most $d$.
\end{defn}
The literature also uses the term $d$-inductive instead of $d$-degenerate.
There are many graph classes that have a bounded degeneracy. For example forests are 1 degenerate graphs. Planar graphs are 5 degenerate while outerplanar graphs are 2 degenerate. There are many more graph classes for which their degeneracy is known.

The definition of degeneracy has a useful equivalence as stated in the following lemma.
\begin{lemma}\label{ordering}
A graph $G=(V,E)$ is $d$-degenerate if and only if it has an ordering of the vertices on a line such that each vertex has at most $d$ neighbors to its left, that means 
\begin{align*}
{|\{j: j<i \wedge \{i,j\}\in E\}| \leq d \qquad \forall i\in[n]}.
\end{align*}
Such an ordering is called Erd\H{o}s-Hajnal sequence.
\end{lemma}
\begin{proof}
Given a $d$-degenerate graph consider $|V|$ positions on the line. Take the vertex in $V$ with degree at most $d$ and put it to the right most position. Subsequently consider the remaining vertex set, which is a subgraph of $G$ and therefore contains again a vertex of degree at most $d$. Again put this vertex to the now right most free position. This vertex has at most $d$ neighbors in the remaining vertex set, which will be completely put in the left neighborhood. In this way we obtain an ordering with the desired property.\\
The other direction is even more straight forward. Given a graph with an ordering of the vertices on a line such that each vertex has at most $d$ neighbors in its left neighborhood the right most vertex in any subgraph has degree at most $d$. Hence the equivalence is proven.
\end{proof}
Let us now state our main theorem.
\begin{thm}\label{degenerate}
Let $G$ be a finite $d$-degenerate graph and $A_0$ be the set of the initially infected vertices of an $r$-bootstrap process with $r\geq d+1$. Then the set $A_f$ of vertices that are infected at the end of the process fulfills
\begin{align*}
|A_0|\leq|A_f|\leq \left(1+\frac{d}{r-d}\right) |A_0|.
\end{align*}   
\end{thm}
\begin{proof}
We introduce a potential $\Psi$, which in each step of the infection process bounds the number of vertices that might be infected in the next step from above. Due to Lemma \ref{ordering} there exists an ordering of the vertices where each vertex has at most $d<r$ neighbors to its left. Since $d<r$, every vertex that becomes infected at some point must have at least one already infected vertex to its right. Think of an enumeration of the vertex infections in the process, where an infection that occurs before another one has a smaller number. Here the specific ordering of the vertices that become infected at the same time does not matter. Now after the $i$th infection, define $\Psi_i$ to be the sum of the number of uninfected vertices in the left neighborhood of each vertex that is currently infected. Note that an uninfected vertex might be counted twice or even more often in our potential if it has more than one infected vertex in its right neighborhood. So $\Psi_0 \leq |A_0|\cdot d$ since every initially infected vertex has at most $d$ uninfected neighbors to its left.
\begin{clm}
The potential decreases after one infection by at least $r-d$, i.e.~we have that $\Psi_{i-1}-\Psi_{i}\geq r-d$.
\end{clm}
\begin{proof}[Proof of claim]
Consider the $i$th infected vertex $v$. Due to the degeneracy, $v$ has at most $d$ uninfected neighbors in its left neighborhood and therefore $v$ can increase the potential by at most $d$. Observe next that for vertex $v$ to become infected it must have at least $r$ infected neighbors. Now there are two kinds of such infecting vertices, namely those that lie in the right neighborhood of $v$ and those that lie in the left neighborhood of $v$. For each infected vertex that lies in the right neighborhood the potential decreases by one, since the vertex $v$ is no longer uninfected and therefore does not contribute to the potential. Each infected vertex in the left neighborhood does not add to the potential either, since it is already infected. Since $r>d$ the potential decreases by at least $r-d$ after each infection.
\end{proof}
Once the potential is less than $1$ the process stops, because in this case no uninfected vertex would have an infected vertex in its right neighborhood and hence could be infected, as $r>d$. Using the claim and the fact that $\Psi_0\leq d|A_0|$ we get
\begin{align*}
|A_f|\leq |A_0|+ \frac{\Psi_0}{r-d}\leq |A_0|+\frac{d |A_0|}{r-d}=\left(1+\frac{d}{r-d}\right) |A_0|.
\end{align*}
It is obvious that $|A_0| \leq |A_f|$ which finishes the proof of Theorem \ref{degenerate}.
\end{proof}
It is also remarkable that our bound on the size of $A_f$ is sharp in a sense: For every $d\geq 1$, $r\geq 
2$ and $\varepsilon>0$ there exists a $d$-degenerate graph with an initially infected set $A_0$ such that the set $A_f$ at the end of the process fulfills ${|A_f|\geq (1-\varepsilon)\left(1+\tfrac{d}{r-d}\right) |A_0|}$. For the construction of such a graph take a set $H$ of $d$ vertices. Consider for some arbitrary natural number $k$, pairs $(U_i,I_i)$ of sets of vertices with $|U_i|=d$ and $|I_i|=r-d$ with $i\in [k]$. Now every vertex in $I_i$ is adjacent to all vertices in $U_i$. Furthermore every vertex in $H$ is adjacent to all vertices in $U_i$ for all $i\in [k]$. Finally choose $A_0=H \cup (\bigcup_{i=1}^k I_i)$. It is easy to see that this graph is a $d$-degenerate graph by construction. A schematic picture of the graph can been seen in figure \ref{examplegraph}

\begin{figure}[H]
\centering
\includegraphics[scale=1]{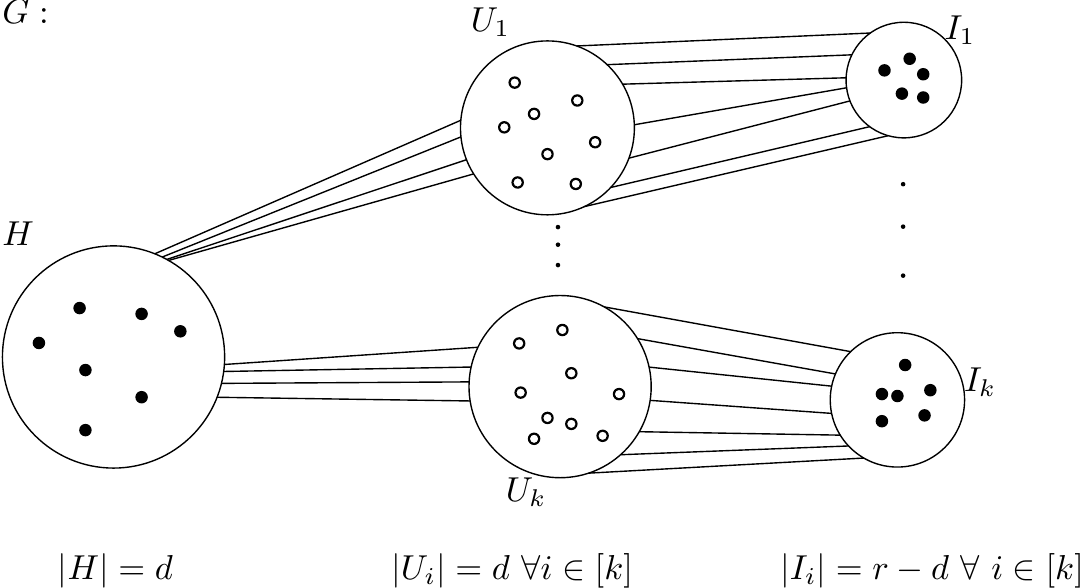}
\caption{Schematic picture of the constructed graph}\label{examplegraph}
\end{figure}

The set $A_0$ is of size $d+k(r-d)$ while the set $A_f$, since every vertex will be infected at the end, is of size $d+k(r-d)+kd=d+kr$. Therefore 
\begin{align*}
\frac{|A_f|-|A_0|}{|A_0|}= \frac{d+kr-(d+k(r-d))}{d+k(r-d)}=\frac{kd}{d+k(r-d)}=\frac{d}{\tfrac{d}{k}+r-d},
\end{align*}
which tends to $\frac{d}{r-d}$ for $k$ tending to infinity. Thus one can choose $k$ large enough such that $|A_f|$ is arbitrarily close to the given upper bound.\\
We would also like to mention that the condition $d\leq r+1$ cannot be dropped due to the fact that the case $d\geq r$ allows the graph to contain a set $H$ of $r$ vertices that are adjacent to all other vertices. Then $H=A_0$ with $|A_0|=r$ suffices to infect the complete graph.\\
Let us now state some Corollaries
\begin{cor}
	Given a $d$-degenerate graph on $n$ vertices. Then the size of a minimal percolating set $A_0$ in the $r$-bootstrap process with $r\geq 2$ fulfills
	\begin{align*}
	 n \cdot \frac{r-d}{r}\leq \vert A_0 \vert.
	\end{align*}
\end{cor}
This is simply obtained by setting $\vert A_f\vert=n$ in Theorem \ref{degenerate}.\\ 
Note that every tree is a $1$-degenerate graph, since every subgraph of a tree is a forest and thus contains at least one leaf, which is a vertex of degree one. As mentioned in the introduction, Riedl considered the size of minimal percolating sets on trees and proved in \cite{riedl} that a minimal percolating set is at least of size $\tfrac{r-1}{r}n+\tfrac{1}{r}$. Our bound yields asymptotically the same bound. Also the construction given in figure \ref{examplegraph} suggests that this bound is also tight in a sense that there exist $d$-degenerate graphs for which the minimal percolating set $A_0$ is of size $n (\tfrac{r-d}{r}+\varepsilon)$.
\begin{cor}
Given a forest on $n$ vertices. Then the size of $A_f$ in the $r$-bootstrap process with a given set $A_0$ and $r\geq 2$ fulfills
\begin{align*}
\vert A_0\vert \leq \vert A_f \vert \leq \frac{r}{r-1}\vert A_0\vert.
\end{align*}
\end{cor}
As mentioned above, every forest is a $1$-degenerate graph and thus the corollary follows from Theorem \ref{degenerate}.
\begin{cor}
	Given a $d$-degenerate graph on $n$ vertices. Then the running time $\tau$ of the $r$-bootstrap process with a given set $A_0$ and $r> d$ is bounded by
	\begin{align*}
	\tau \leq  \frac{d}{r-d} \vert A_0 \vert.
	\end{align*}
\end{cor}
This again follows immediately from Theorem \ref{degenerate}, as every additional infection takes at most one time step.

\bibliographystyle{amsplain}
\bibliography{bibliography}{}

\end{document}